\documentclass{article}
\usepackage{amsmath}
\usepackage{amssymb}
\usepackage{amsfonts}
\usepackage[top=1.5in,bottom=1in,left=1.25in,right=1.25in]{geometry}

\newtheorem{theorem}{Theorem}
\newtheorem{acknowledgement}[theorem]{Acknowledgement}

\newtheorem{definition}[theorem]{Definition}

\newtheorem{lemma}[theorem]{Lemma}

\newtheorem{proposition}[theorem]{Proposition}
\newtheorem{remark}[theorem]{Remark}

\newenvironment{proof}[1][Proof]{\noindent\textbf{#1.} }{\ \rule{0.5em}{0.5em}}

\begin{document}

\title{Krichever-Novikov Vertex Algebras on Compact Riemann Surfaces}
\author{Lu Ding\thanks{%
Email: dinglu@amss.ac.cn} \ \ \ \ \ Shikun Wang\thanks{%
Email: wsk@amss.ac.cn} \\
Institute of Applied Mathematics, Academy of Mathematics and Systems
Science, \\
Chinese Academy of Sciences, Beijing 100080, People's Republic of China}
\date{}
\maketitle

\begin{abstract}
We give a notation of Krichever-Novikov vertex algebras on compact Riemann
surfaces which is a bit weaker, but quite similar to vertex algebras. As
example, we construct Krichever-Novikov vertex algebras of generalized
Heisenberg algebras on arbitrary compact Riemann surfaces, which are reduced
to be Heisenberg vertex algebra when restricted on Riemann spheres.
\end{abstract}

\section{Introduction\label{introduction}}

\bigskip \bigskip Vertex operator algebras are a class of algebras, whose
structures arose naturally from vertex operator constructions of
representations of affine Lie algebras and in the work of
Frenkel-Lepowsky-Meurman and Borcherds on the \textquotedblleft moonshine
module\textquotedblright\ for the Monster finite simple group \cite{B1} \cite%
{B2} \cite{CN} \cite{FLM1} \cite{FLM3}. They are (rigorous) mathematical
counterparts of chiral algebras in 2-dimensional conformal field theory in
physics. A vertex algebra consists in principle of a state-field
correspondence $V\rightarrow End\left( V\right) \left[ \left[ z,z^{-1}\right]
\right] $ that maps any element $a$ to a field $Y\left( a,z\right)
=\sum_{n\in
\mathbb{Z}
}a_{n}z^{n}$ with $a_{n}\in End\left( V\right) $ which satisfies vacuum
axiom, translation axiom and locality axiom. Typical examples of vertex
algebras are the Heisenberg vertex algebra, affine Kac-Moody vertex algebra
and lattice vertex algebra.

The goal of the present paper is the construction, as we hope, of regular
analogus of vertex algebras, connected with Riemann surfaces of genus $g\geq
0.$ It's not surprising that the space $End\left( V\right) \left[ \left[
z,z^{-1}\right] \right] $ should be replaced by a certain space which
encodes the geometry of the Riemann surface. \bigskip We consider its
important subspace-the associative algebra $%
\mathbb{C}
\left[ z,z^{-1}\right] $ of Laurent polynomials firstly. According to the
view of Krichever and Novikov (\cite{KN1}, \cite{KN2}, \cite{KN3}, \cite{L1}%
), if we identify it with the algebra of meromorphic functions on Riemann
sphere which are holomorphic outside zero and infinite, then it is easily
generalized on higher genus Riemann surfaces to the associative algebra of
meromorphic functions which are holomorphic outside two distinguished
points. \cite{KN1} cited that the new associative algebra has a basis $%
\{A_{n}\left( P\right) |n\in Z^{\prime }\}$ by Riemann-Roch theorem, where $%
Z^{\prime }$ takes the integer set or half-integer set depending on whether
the genus $g$ is even or odd. If we define $\left( A\left( P\right) \right)
^{n}$ by $A_{n+\frac{g}{2}}\left( P\right) ,$ then it is the space $%
\mathbb{C}
\left[ A\left( P\right) ,A\left( P\right) ^{-1}\right] $ which is the
extension of the algebra of Laurent polynomials. Hence, the space $End\left(
V\right) \left[ \left[ z,z^{-1}\right] \right] $ is naturally extended on
any higher genus Riemann surface to $End\left( V\right) \left[ \left[
A\left( P\right) ,A\left( P\right) ^{-1}\right] \right] $ in which elements
are the form $\sum_{n\in
\mathbb{Z}
}a_{n}\left( A\left( P\right) \right) ^{n}$ or $\sum_{n\in Z^{\prime
}}a_{n}A_{n}\left( P\right) .$ Obviously, the axioms in vertex algebras need
to be reformulated in a suitable way as well. Here, we consider the locality
axiom only. Note that the locality axiom%
\begin{equation*}
\left( z-w\right) ^{N}\left[ Y\left( a,z\right) ,Y\left( b,w\right) \right]
=0,N\gg 0
\end{equation*}%
is equivalent to
\begin{equation*}
\left( z^{n}-w^{n}\right) ^{N}\left[ Y\left( a,z\right) ,Y\left( b,w\right) %
\right] =0,\text{ }\forall n\in
\mathbb{Z}
,\text{ }N\gg 0.
\end{equation*}%
As explained above, on a Reimann surface, the status of $z^{n}$ can be
replaced by some meromorphic function $A_{m}\left( P\right) $. Then the
locality axiom of KN\ vertex algebra is reformulated as
\begin{equation*}
\left[ A_{m}\left( P\right) -A_{m}\left( Q\right) \right] ^{N}\left[ Y\left(
a,P\right) ,Y\left( b,Q\right) \right] =0,\text{ }\forall m\in Z^{\prime },%
\text{ }N\gg 0.
\end{equation*}

Analogously, the derivative, the function $\frac{1}{z-w}$\bigskip $,$ and
distance function are reformulated as Lie derivative, Szeg$\ddot{o}$ kernel
and the function defined by the level line respectively.

Using the data above, we can give the extention of the vertex algebra which
is called the Krichever-Novikov vertex algebra with respect to a compact
Riemann surface and the two distinguished points in general positions. For
briefly, we also call it the KN\ vertex algebra if no confusion.

As examples of KN vertex algebras, we construct KN vertex algebras of
generalized Heisenberg algebras on compact Reimann surfaces. When restricted
on Riemann spheres, they are the Heisenberg vertex algebras.

The paper is organized as follows. Section 2 gives a brief reviews of KN
basis and results needed later in order to make this paper self-contained.
Section 3 sets up the notations and gives the definition of vertex algebras
on compact Riemann surfaces. In section 4, we construct the Heisenberg
vertex algebras on Riemann surfaces. Also, we introduce the Szeg$\ddot{o}$
kernel and the Level line used to prove the data we constructed satisfy the
axioms of a vertex algebra on a Riemann surface. In the last section, we
will see that the generalized Heisenberg vertex algebras on Riemann surfaces
with genus zero is the usual Heisenberg vertex algebra.

In the forthcoming papers, Kac-Moody KN vertex algebra and Virasoro KN
vertex algebra on a Riemann surface will be given. And the structures of KN
vertex operator algebras on Riemann surfaces will be discussed later.

\section{Notation}

In this section, we recall some results on Riemann surfaces, and then reduce
some results which are needed later.

\subsection{Krichever-Novikov bases\label{KN basis in notation}}

Let $\mathcal{M}$ be a compact Riemann surface of genus $g,$ $S_{+},S_{-}$
two distinguished points in general position.\bigskip\ The KN bases \cite%
{KN1} \cite{KN2} \cite{KN3} are certain bases for the spaces $\mathcal{F}%
^{\lambda }$ of meromorphic tensors of weight $\lambda $ on the Riemann
surface $\mathcal{M}$ which are holomorphic outside $S_{+}$ and $S_{-}.$

For integer $\lambda \not=0,1$ and $g>1,$ the Riemann-Roch theorem
guarantees the existence and uniqueness of meromorphic tensors of conformal
weight $\lambda $ which are holomorphic outside $S_{+}$ and $S_{-},$ and
have the following behavior in a neighborhood of $S_{+}$ and $S_{-}$:%
\begin{equation}
f_{\lambda ,n}\left( z_{\pm }\right) =\varphi _{\lambda ,n}^{\pm }z_{\pm
}^{\pm n-s_{\lambda }}\left( 1+O\left( z_{\pm }\right) \right) \left(
dz_{\pm }\right) ^{\lambda },  \label{formula '11}
\end{equation}%
where $s_{\lambda }=\frac{g}{2}-\lambda \left( g-1\right) ,$ $\varphi
_{\lambda ,n}^{+}=1,\varphi _{\lambda ,n}^{-}\not=0.$ Here $\left( dz_{\pm
}\right) ^{\lambda }$ means $\left( \frac{\partial }{\partial z_{\pm }}%
\right) ^{-\lambda }$ for $\lambda <0,$ $z_{+}$ and $z_{-}$ are local
coordinates at small neighborhoods of $S_{+}$ and $S_{-}$ respectively which
satisfy $z_{+}\left( S_{+}\right) =0$ and $z_{-}\left( S_{-}\right) =0.$ The
index $n$ in Eq. $\left( \ref{formula '11}\right) $ takes either integer or
half-integer values depending on whether $g$ is even or odd.

For $\lambda =0,$ the behavior is modified with respect to Eq. $\left( \ref%
{formula '11}\right) .$ Let $A_{n},$ $\left\vert n\right\vert \geq \frac{g}{2%
}+1,$ be the unique function which has the Laurent expansion in a
neighborhood of $S_{\pm }$:%
\begin{equation}
A_{n}\left( z_{\pm }\right) =\alpha _{n}^{\pm }z_{\pm }^{\pm n-\frac{g}{2}%
}\left( 1+O\left( z_{\pm }\right) \right) ,  \label{formula'12}
\end{equation}%
where $\alpha _{n}^{+}=1,\alpha _{n}^{-}$ is some nonzero complex number. As
before, $n$ is integer or half-integer depending on the parity of $g$. For $%
n=-\frac{g}{2},...,\frac{g}{2}-1$ we take the function with the following
behavior!!!
\begin{eqnarray}
A_{n}\left( z_{+}\right) &=&\alpha _{n}^{+}z_{+}^{n-\frac{g}{2}}\left(
1+O\left( z_{+}\right) \right)  \label{3} \\
A_{n}\left( z_{-}\right) &=&\alpha _{n}^{-}z_{-}^{-n-\frac{g}{2}{-1}}\left(
1+O\left( z_{-}\right) \right) ,  \notag
\end{eqnarray}%
where $\alpha _{n}^{+},\alpha _{n}^{-}$ are required as before. For $n=\frac{%
g}{2},$ choose $A_{\frac{g}{2}}=1.$

For $\lambda =1,$ we take the basis of one-forms as follows: in the range $%
\left\vert n\right\vert \geq \frac{g}{2}+1,$ $\omega ^{n}=f_{1,-n}$ with $%
f_{1,-n}$ given by $\left( \ref{formula '11}\right) ;$ for $n=-\frac{g}{2}%
,...,\frac{g}{2}-1,$ those is specified by the local series
\begin{eqnarray}
\omega ^{n}\left( z_{+}\right) &=&\beta _{n}^{+}z_{+}^{-n+\frac{g}{2}{-1}%
}\left( 1+O\left( z_{+}\right) \right) dz_{+}\text{ }  \label{formula '14} \\
\omega ^{n}\left( z_{-}\right) &=&\beta _{n}^{-}z_{-}^{-n+\frac{g}{2}}\left(
1+O\left( z_{-}\right) \right) dz_{-},\text{ }  \notag
\end{eqnarray}%
Here, choose $\beta _{n}^{+}=1$ to fix $\omega ^{n}$ \ and take $\omega ^{%
\frac{g}{2}}$ as the Abelian differential of the third kind with simple
poles in $S_{\pm }$ and residues $\pm 1,$ normalized in such a way that its
periods over all cycles be purely imaginary.

In the case $g=1$, the existence of a nonzero holomorphic one-form $\xi $
with $\xi \left( z_{+}\right) =\left( 1+o\left( z_{+}\right) \right) dz_{+}$
enables us to construct a series of $\lambda -$forms which are holomorphioc
outside $S_{\pm }$ by zero-forms:%
\begin{equation*}
f_{\lambda ,n}=A_{n}\xi ^{\lambda },
\end{equation*}%
where the $A_{n}$'s are defined by Equations $\left( \ref{formula'12}\right)
$ and $\left( \ref{3}\right) $.

\bigskip For $g=0$, when the compact Riemann surface is the ordinary
completion of the complex plane, and $S_{\pm }$ are the points $z=0$ and $%
z=\infty ,$ then the functions $A_{n}$ concide with $z^{n},$ the generators
of the Laurent basis, and $f_{\lambda ,n}\left( z\right) =z^{n-\lambda
}\left( dz\right) ^{\lambda }$ for $\lambda \in
\mathbb{Z}
.$

Let $f_{0,n}=A_{n},$ $f_{1,n}=\omega ^{-n}$ and $Z^{\prime }=%
\mathbb{Z}
+\frac{g}{2}.$ On any compact Riemann surface, by Riemann-Rock theorem, $%
\left\{ f_{\lambda ,n}|n\in Z^{\prime }\right\} $ makes up of a basis of $%
\mathcal{F}^{\lambda }\ $which is called the Krichever-Novikov basis (KN
basis) with weight $\lambda $ $\left( \cite{KN1}-\cite{KN3}\right) .$

Since all small cycles are homologous around $S_{+}\left( \text{
respectively, }S_{-}\right) $ and $\omega \cdot \eta $ is holomorphic away
from $S_{+}$ and $S_{-}$ for any $\omega \in \mathcal{F}^{\lambda },$ $\eta
\in \mathcal{F}^{1-\lambda },$ then the integral around the two points
doesn't depend on the choice of cycles. Hence, we can define residue
operator by integral as follows.

\begin{definition}
\begin{equation*}
\begin{array}{llllll}
Res_{S_{\pm }}\mathtt{:} & \mathcal{F}^{\lambda } & \times & \mathcal{F}%
^{1-\lambda } & \longrightarrow &
\mathbb{C}
\\
& \omega & , & \eta & \mapsto & \frac{1}{2\pi i}\oint_{c_{s_{\pm }}}\omega
\cdot \eta ,%
\end{array}%
\end{equation*}%
where $c_{s_{+}}\left( c_{s_{-}}\right) $ is a cycle homologous to a small
cocycle around the points $S_{+}\left( S_{-}\right) .$
\end{definition}

For simplicity, we denote $Res_{S_{+}}$ by $Res$ if there is no confusion.
Define $f_{\lambda }^{n}$ by $f_{\lambda ,-n}$ with $n\in Z^{\prime },$ $%
\delta _{m}^{n}$ by $1$ if $n=m$ and $0$ otherwise as usual. By direct
calculus, we have the following important proposition.

\begin{proposition}
\label{res(fnfm)=deltamn}$Res\left[ f_{\lambda ,n}\left( P\right)
f_{1-\lambda }^{m}\left( P\right) \right] =\delta _{n}^{m},$ where $n,m\in
Z^{\prime }.$
\end{proposition}

\subsection{Lie derivative}

\begin{definition}
Define Lie derivative of a tensor field $g\left( P\right) $ with respect to
a tangent verctor field $\zeta \left( P\right) $ locally as%
\begin{eqnarray*}
\nabla _{\zeta \left( P\right) }g\left( P\right) |_{U\left( P\right) }
&:&=\nabla _{\zeta \left( z\right) \frac{d}{dz}}\left( g\left( z\right)
dz^{\lambda }\right) \\
&=&\left( \zeta \left( z\right) \frac{dg\left( z\right) }{dz}+\lambda
g\left( z\right) \frac{d\zeta \left( z\right) }{dz}\right) \left( dz\right)
^{\lambda },
\end{eqnarray*}%
where $z$ is a local coordinate of the neighborhood $U\left( P\right) $ with
$P\in \mathcal{M}.$
\end{definition}

\begin{definition}
Denote $f_{-1,\frac{3g}{2}-1}\left( P\right) $ by $e_{\frac{3g}{2}-1}\left(
P\right) $ or $e\left( P\right) ,$ and $\nabla _{e_{\frac{3g}{2}-1}\left(
P\right) }$ by $\nabla .$ Let $\left\langle ,\right\rangle $ be the coupling
of 1-form and tangent vector field.
\end{definition}

\begin{lemma}
\label{lie derivative on 1-form}If $g\left( P\right) \in \mathcal{F}%
^{\lambda }$, then

\begin{enumerate}
\item $\nabla \left\langle g\left( P\right) ,e\left( P\right) \right\rangle
=\left\langle \nabla g\left( P\right) ,e\left( P\right) \right\rangle ;$

\item $\nabla g\left( P\right) =d\left\langle g\left( P\right) ,e\left(
P\right) \right\rangle .$
\end{enumerate}
\end{lemma}

\begin{proposition}
\label{lie derivative on omiga}$\nabla \omega ^{n}\left( P\right)
=\sum_{m\in Z^{\prime }}\xi _{m}^{n}\omega ^{m},$ where $\xi
_{m}^{n}=-Res\left( \left\langle \omega ^{n}\left( P\right) ,e_{\frac{3g}{2}%
-1}\left( P\right) \right\rangle dA_{m}\left( P\right) \right) .$
\end{proposition}

\begin{proof}
By the second part of Lemma \ref{lie derivative on 1-form} and Proposition %
\ref{res(fnfm)=deltamn}, we have%
\begin{eqnarray*}
\xi _{m}^{n} &=&Res\left( \nabla \omega ^{n}\left( P\right) A_{m}\left(
P\right) \right) \\
&=&Res\left( d\left\langle \omega ^{n}\left( P\right) ,e\left( P\right)
\right\rangle A_{m}\left( P\right) \right) \\
&=&Res\left( d\left[ \left\langle \omega ^{n}\left( P\right) ,e\left(
P\right) \right\rangle A_{m}\left( P\right) \right] -\left\langle \omega
^{n}\left( P\right) ,e\left( P\right) \right\rangle dA_{m}\left( P\right)
\right) .
\end{eqnarray*}%
Note that the first term is zero, we complete the proof.
\end{proof}

\begin{lemma}
\label{prop183}For any $n\in Z^{\prime },$ $\sum_{u_{1},u_{2},\cdots
,u_{k-1}}\xi _{u_{1}}^{n}\xi _{u_{2}}^{u_{1}}\cdots \xi _{u_{k}}^{u_{k-1}}$
is

\begin{enumerate}
\item $\left( -n+\frac{g}{2}-1\right) \cdots \left( -n+\frac{g}{2}-k\right)
, $ if $u_{k}=n+k$ $;$

\item $0,$ if $u_{k}>n+k;$

\item $0,$ if $n=\frac{g}{2}-k,\cdots ,\frac{g}{2}-1.$
\end{enumerate}
\end{lemma}

\begin{proof}
By the part (1) of Lemma \ref{lie derivative on 1-form} and Propostion \ref%
{lie derivative on omiga}, we have%
\begin{equation*}
\nabla ^{k}\left\langle \omega ^{n}\left( P\right) ,e\left( P\right)
\right\rangle =\xi _{u_{1}}^{n}\xi _{u_{2}}^{u_{1}}\cdots \xi
_{u_{k}}^{u_{k-1}}\left\langle \omega ^{u_{k}}\left( P\right) ,e\left(
P\right) \right\rangle ,
\end{equation*}%
Here, and henceforth, repeated indices are summed in the interger or
half-integer set depending on the parity of the genus of the Riemian
surface. From Section \ref{KN basis in notation}, we know the equation above
has the following local behaviour near the point $S_{+}$
\begin{equation*}
\left( \left( 1+o\left( z\right) \right) \frac{\partial }{\partial z}\right)
^{k}\left[ z^{-n+\frac{g}{2}-1}\left( 1+o\left( z\right) \right) \right]
=\xi _{u_{1}}^{n}\xi _{u_{2}}^{u_{1}}\cdots \xi _{u_{k}}^{u_{k-1}}z^{-u_{k}+%
\frac{g}{2}-1}\left[ 1+o\left( z\right) \right] .
\end{equation*}%
That is,%
\begin{equation*}
\left( -n+\frac{g}{2}-1\right) \cdots \left( -n+\frac{g}{2}-k\right) z^{-n+%
\frac{g}{2}-1-k}\left[ 1+o\left( z\right) \right] =\xi _{u_{1}}^{n}\xi
_{u_{2}}^{u_{1}}\cdots \xi _{u_{k}}^{u_{k-1}}z^{-u_{k}+\frac{g}{2}-1}\left[
1+o\left( z\right) \right] .
\end{equation*}%
Compare the coeffients of $z^{m}$ of the two sides of the above equation, we
can get the result.
\end{proof}

\subsection{Szeg$\ddot{o}$ kernel, Level line and formal delta function}

Here, we introduce Szeg$\ddot{o}$ kernel, level line briefly \cite{Bonora}.
By the data, we construct and research on formal delta functions on compact
Riemann surfaces.

Let $\mathcal{M}$ be a compact Riemann surface with genus $g,$ $S_{+}$ and $%
S_{-}$ are distinct points in general position.

\begin{definition}
Define the Szeg$\ddot{o}$ kernel by $S\left( P,Q\right) =\frac{E\left(
P,P_{+}\right) \theta \left( P-Q-u\right) \theta \left( Q-P_{+}-u\right) }{%
E\left( P,Q\right) E\left( Q,P_{+}\right) \theta \left( u\right) \theta
\left( P-P_{+}-u\right) },$ where $E\left( P,Q\right) $ is the
Schottky-Klein prime form, $u=gP_{-}-P_{+}-\bigtriangleup .$
\end{definition}

$S\left( P,Q\right) $ is a holomorphic $\left( 0,1\right) $ form outside $%
P=Q $ on $\mathcal{M}\times \mathcal{M}.$ In a neighborhood of $P=Q,$
\begin{equation*}
S\left( P,Q\right) |_{U\left( P=Q\right) }=\frac{1}{z\left( P\right)
-z\left( Q\right) }\left[ 1+o\left( z\left( P\right) -z\left( Q\right)
\right) ^{2}\right] dz\left( Q\right) .
\end{equation*}

Suppose $\rho $ is the unique differential of the third kind with pure
imaginary periods, that is, $\rho $ has poles of order $1$ at $S_{\pm },$ $%
Res_{S_{\pm }}\rho =\pm 1,$ and has purely imaginary period. Let $P_{0}$ be
an arbitrary reference point on $\mathcal{M}$ different from $S_{\pm }.$ The
contours $c_{\tau }$ are defined as the level lines of the function, i.e.
\begin{equation*}
c_{\tau }=\{P\in \mathcal{M}|Re\int_{P_{0}}^{P}\rho =\tau \}.
\end{equation*}%
For $\tau \rightarrow \pm \infty ,$ the contours become small circles around
$S_{\mp }.$

For $g=0,$ the Szeg$\ddot{o}$ kernel is $S\left( z,w\right) =\frac{1}{z-w}%
dw, $ the level lines defined by $Re\int_{P_{0}}^{P}\frac{1}{z}dz$ are
circles around the origin in the complex plane, where $P_{0}\in
\mathbb{C}
\backslash \{0\}.$ In the regions of $\tau \left( z\right) >\tau \left(
w\right) $ and $\tau \left( w\right) >\tau \left( z\right) $ $,$ the Szeg$%
\ddot{o}$ kernel has expansions $\Sigma _{n<0}z^{n}w^{-n-1}dw$ and $-\Sigma
_{n\geq 0}z^{n}w^{-n-1}dw$ respectively. Their minus is $\Sigma _{n\in
\mathbb{Z}
}z^{n}w^{-n-1}dw$ which is similar to the formal delta function $\delta
\left( z,w\right) =\Sigma _{n\in
\mathbb{Z}
}z^{n}w^{-n-1}.$

For higher genus Riemann surfaces, we hope to get the formal delta functions
using the similar way. Denote the expansion of $S\left( P,Q\right) $ in $%
\tau \left( Q\right) >\tau \left( P\right) $ and $\tau \left( P\right) >\tau
\left( Q\right) $ by $i_{Q,P}S\left( P,Q\right) $ and $i_{P,Q}S\left(
P,Q\right) $ respectively. In \cite{KN1}-\cite{KN3} and \cite{wangshikun},
we know
\begin{equation}
i_{P,Q}S_{1}\left( P,Q\right) =\sum_{n<\frac{g}{2}}A_{n}\left( P\right)
\omega ^{n}\left( Q\right) ,\text{ }i_{Q,P}S_{1}\left( P,Q\right)
=-\sum_{n\geq \frac{g}{2}}A_{n}\left( P\right) \omega ^{n}\left( Q\right) .
\label{19}
\end{equation}%
Denote their minus by $\triangle \left( P,Q\right) .$ Clearly, $\triangle
\left( P,Q\right) =\sum_{n\in Z^{\prime }}A_{n}\left( P\right) \omega
^{n}\left( Q\right) .$ We call it the formal delta function on the Riemann
surface not only because the way of its construction is similar to the
formal delta function $\delta \left( z,w\right) $'s, but also it has the
complete analogous property of delta function.

\begin{proposition}
If $f\in \mathcal{F}^{0},g\in \mathcal{F}^{1},$ then
\begin{equation*}
Res_{Q=S_{+}}\left( \triangle \left( P,Q\right) f\left( Q\right) \right)
=f\left( P\right) ,\text{ }Res_{Q=S_{+}}\left( \triangle \left( P,Q\right)
g\left( P\right) \right) =g\left( Q\right) .
\end{equation*}
\end{proposition}

\begin{proof}
It's easy to prove by Proposition \ref{res(fnfm)=deltamn}.
\end{proof}

Recall in vertex algebras, the following properties of the formal delta
function is very useful in proving the locality axiom
\begin{equation}
\partial _{z}\delta \left( z,w\right) =-\partial _{w}\delta \left(
z,w\right) ,\text{ }\left( z^{n}-w^{n}\right) ^{m}\partial _{z}^{m+1}\delta
\left( z,w\right) =0.  \label{properties of formal delta function}
\end{equation}%
From the below Propositions, we know, as we hope, the formal delta function
on the Riemann surface also has the analogous properties in which $\partial
_{z}$ and $z^{n}$ are replaced by $\nabla _{P}$ and $A_{n}\left( P\right) $
respectively.

\begin{proposition}
$\nabla _{P}\bigtriangleup \left( P,Q\right) =-\nabla _{Q}\bigtriangleup
\left( P,Q\right) .$
\end{proposition}

\begin{proof}
By Proposition \ref{res(fnfm)=deltamn}, $\nabla \omega ^{n}\left( P\right)
=\Sigma \xi _{u}^{n}\omega ^{u}\left( P\right) ,$ $\nabla A_{n}\left(
P\right) =\Sigma \eta _{n}^{u}A_{u}\left( P\right) ,$ where $\xi
_{u}^{n}=Res\left( \nabla \omega ^{n}\left( P\right) A_{u}\left( P\right)
\right) ,$ $\eta _{u}^{n}=Res\left( \nabla A_{u}\left( P\right) \omega
^{n}\left( P\right) \right) .$ Note that $Res\left[ \nabla \left(
A_{u}\left( P\right) \omega ^{n}\left( P\right) \right) \right] =0,$ then $%
\xi _{u}^{n}=-\eta _{u}^{n}.$ Hence%
\begin{eqnarray*}
\nabla _{P}\bigtriangleup \left( P,Q\right) &=&\left[ \nabla A_{n}\left(
P\right) \right] \omega ^{n}\left( Q\right) =\eta _{n}^{u}A_{u}\left(
P\right) \omega ^{n}\left( Q\right) =-\xi _{n}^{u}A_{u}\left( P\right)
\omega ^{n}\left( Q\right) \\
&=&-A_{u}\left( P\right) \nabla \omega ^{u}\left( Q\right) =-\nabla
_{Q}\bigtriangleup \left( P,Q\right) .
\end{eqnarray*}
\end{proof}

\begin{proposition}
\label{1121}$\left[ A_{u}\left( P\right) -A_{u}\left( Q\right) \right]
\triangle \left( P,Q\right) =0$ with $u\in Z^{\prime }.$
\end{proposition}

\begin{proof}
Since $A_{u}\left( P\right) \cdot $ $A_{n}\left( P\right) \in \mathcal{F}%
^{0} $ and $\{A_{n}\left( P\right) |$ $n\in Z^{\prime }\}$ is a basis of $%
\mathcal{F}^{0},$ we can suppose that
\begin{equation*}
A_{u}\left( P\right) A_{n}\left( P\right) =\sum_{m\in Z^{\prime }}\alpha
_{un}^{m}A_{m}\left( P\right) ,
\end{equation*}%
where $\alpha _{un}^{m}\in
\mathbb{C}
.$ Note that $Res\left( A_{m}\left( P\right) \omega ^{k}\left( P\right)
\right) =\delta _{m}^{k},$ we get $\alpha _{un}^{m}=Res\left[ A_{u}\left(
P\right) A_{n}\left( P\right) \omega ^{m}\left( P\right) \right] .$ Using
the same method, we have
\begin{equation*}
A_{u}\left( Q\right) \omega ^{n}\left( Q\right) =\alpha _{um}^{n}\omega
^{m}\left( Q\right) .
\end{equation*}%
Hence%
\begin{eqnarray*}
&&\left[ A_{u}\left( P\right) -A_{u}\left( Q\right) \right] \triangle \left(
P,Q\right) \\
&=&\left[ A_{u}\left( P\right) -A_{u}\left( Q\right) \right] A_{n}\left(
P\right) \omega ^{n}\left( Q\right) \\
&=&\left[ A_{u}\left( P\right) A_{n}\left( P\right) \right] \omega
^{n}\left( Q\right) -A_{n}\left( P\right) \left[ A_{u}\left( Q\right) \omega
^{n}\left( Q\right) \right] \\
&=&\alpha _{un}^{m}A_{m}\left( P\right) \omega ^{n}\left( Q\right)
-A_{n}\left( P\right) \alpha _{um}^{n}\omega ^{m}\left( Q\right) =0.
\end{eqnarray*}
\end{proof}

Differentiating the formula in the above proposition with resepct to $P,$
and multiplying the result by $\left( A_{u}\left( P\right) -A_{u}\left(
Q\right) \right) ,$ we obtain
\begin{equation}
\left[ A_{u}\left( P\right) -A_{u}\left( Q\right) \right] ^{2}d_{P}\triangle
\left( P,Q\right) =0,  \label{1122}
\end{equation}%
where $d_{P}\triangle \left( P,Q\right) $ means $d\left( A_{n}\left(
P\right) \right) \omega ^{n}\left( Q\right) .$

\begin{theorem}
\label{214}For any $u\in Z^{\prime },m,n\in
\mathbb{Z}
_{+},$ $\left[ A_{u}\left( P\right) -A_{u}\left( Q\right) \right]
^{m+n+2}\left( \nabla _{P}^{m}\nabla _{Q}^{n}d_{P}\triangle \left(
P,Q\right) \right) =0$ on $\mathcal{M}\times \mathcal{M}.$
\end{theorem}

\begin{proof}
By Eq. \ref{1122}, we assume it's true when $m+n=\tilde{N}$ with $\tilde{N}%
\in
\mathbb{Z}
_{+}.$ We will prove the theorem by induction. When $m+n=\tilde{N}+1,$ there
exists a pair of non-negative integers $\left( M,N\right) $ such that $%
\left( m,n\right) =\left( M+1,N\right) $ or $\left( M,N+1\right) .$ We only
prove the theorem for the first case, the other case can be got by the
complete way. Note that $\left[ A_{u}\left( P\right) -A_{u}\left( Q\right) %
\right] ^{M+N+3}\left( \nabla _{P}^{M+1}\nabla _{Q}^{N}d_{P}\triangle \left(
P,Q\right) \right) $ is equal to
\begin{equation*}
\nabla _{P}\left( \left[ A_{u}\left( P\right) -A_{u}\left( Q\right) \right]
^{M+N+3}\nabla _{P}^{M}\nabla _{Q}^{N}d_{P}\triangle \left( P,Q\right)
\right) -\left( \nabla _{P}\left[ A_{u}\left( P\right) -A_{u}\left( Q\right) %
\right] ^{M+N+3}\right) \nabla _{P}^{M}\nabla _{Q}^{N}d_{P}\triangle \left(
P,Q\right) .
\end{equation*}%
By our inductive assumption, the first term is zero. The second term is
\begin{equation*}
-\left( M+N+3\right) \nabla _{P}A_{u}\left( P\right) \left[ A_{u}\left(
P\right) -A_{u}\left( Q\right) \right] ^{M+N+2}\nabla _{P}^{M}\nabla
_{Q}^{N}d_{P}\triangle \left( P,Q\right) .
\end{equation*}%
Using our inductive assumption again, we got the result.
\end{proof}

\section{KN vertex algebras on compact Riemann surfaces}

\subsection{KN\ Fields and derivatives}

Recall in a vertex algebra, a formal series $a\left( z\right) =\sum_{n\in
\mathbb{Z}
}a_{n}z^{n}$ is a field if for any $v\in V,$ $a_{n}v=0,$ $n\ll 0.$ In other
words, the formal series is a formal Laurent series.

For a higher genus Riemann surface, as we explained in the introduction, the
space $End\left( V\right) \left[ \left[ z,z^{-1}\right] \right] $ is
replaced by $End\left( V\right) \left[ \left[ A\left( P\right) ,A\left(
P\right) ^{-1}\right] \right] ,$ and the formal series $a\left( z\right) $
is replaced by $a\left( P\right) =\sum_{n\in Z^{\prime }}a_{n}A_{n}\left(
P\right) .$ The new formal series is called a KN field if for any $v\in V,$
there exists a number $N$ such that $a_{n}v=0$ for $n\leq N.$

Recall the derivative on the field $a\left( z\right) $ is defined by $%
\partial _{z}$ which acts on functions $z^{n}$ directly. Since $\nabla
_{e\left( P\right) }$ is the extension of $\partial _{z}$ on any Riemann
surface, then the derivative of a KN field $a\left( P\right) $ is defined
according to Lie derivative, that is, $\nabla a\left( P\right) =$ $%
\sum_{n\in Z^{\prime }}a_{n}\nabla A_{n}\left( P\right) .$

\subsection{KN vertex algebras}

Using the above data, we can give the notation of a KN vertex algebra on a
compact Riemann surface.

\begin{definition}
Suppose $\mathcal{M}$ is a compact Riemann surface with genus $g,$ $S_{+}$
and $S_{-}$ are two distinguished points in general positions. \textbf{A
Krichever-Novikov vertex algebra} $\left( \mathcal{M},V,S_{\pm
},T,\left\vert 0\right\rangle \right) $ (KN vertex algebra) is a collection
of data

\begin{enumerate}
\item Space of states: $V$ is a $%
\mathbb{Z}
-$graded vector space (graded by weights)%
\begin{equation*}
V=\bigoplus\limits_{n\in
\mathbb{Z}
}V_{n},\text{ for }v\in V_{n},n=wt\{v\}
\end{equation*}%
where $\dim V_{n}<\infty $ for $n\in
\mathbb{Z}
$ and $V_{n}=0$ for $n$ sufficiently negative;

\item Vacuum vector: $|0\rangle \in V_{0};$

\item Translation operator: a linear operator $T:V\rightarrow V$;

\item KN\ vertex operators: a linear operation%
\begin{equation*}
\begin{array}{lllll}
Y\left( \cdot ,P\right) & : & V & \rightarrow & End\left( V\right) \left[ %
\left[ A^{+}\left( P\right) ,A^{-}\left( P\right) \right] \right]%
\end{array}%
\end{equation*}%
taking each $a\in V$ to a KN field acting on $V,$%
\begin{equation*}
Y\left( a,P\right) =\sum_{n\in Z^{\prime }}a_{n}A_{n}\left( P\right) ,\text{
}Z^{\prime }=%
\mathbb{Z}
+\frac{g}{2}.
\end{equation*}
\end{enumerate}

And these data are subject to the following axioms:

\begin{enumerate}
\item Vacuum axiom: $Y\left( |0\rangle ,P\right) =Id_{V}.$ Furthermore, for
any $a\in V,$ $Y\left( a,P\right) \left\vert 0\right\rangle $ has a
well-defined value at $P=S_{+}.$ When $a$ is a homogeneous element with
weight $k,$ then%
\begin{equation*}
Y\left( a,P\right) |0\rangle |_{P=S_{+}}\equiv a\text{ }mod\text{ }%
\bigoplus\limits_{n<k}V_{n}.
\end{equation*}

\item Translation axiom: $T|0\rangle =0.$ For any nonzero space $V_{n},$
there always exists nonzero element $a$ in $V_{n}$ such that%
\begin{equation*}
\left[ T,Y\left( a,P\right) \right] =\nabla Y\left( a,P\right) .
\end{equation*}

\item Locality axiom: For any $a,b\in V,$ there exists a positive integer $N$
such that
\begin{equation*}
F_{u}\left( P,Q\right) ^{N}\left[ Y\left( a,P\right) ,Y\left( b,Q\right) %
\right] =0,
\end{equation*}%
on $\mathcal{M\times M},$ where $F_{u}\left( P,Q\right) =A_{u}\left(
P\right) -A_{u}\left( Q\right) $ for any $u\in Z^{\prime }.$
\end{enumerate}
\end{definition}

\section{KN vertex algebras of generalized Heisenberg algebras}

In this section, we will construct a KN vertex algebras of the generalized
Heisenberg\ algebra on an arbitrary compact Riemann surface as an example.
Firstly, we recall the definition of the generalized Heisenberg algebra (%
\cite{KN1}-\cite{KN3}).

\begin{definition}
By a \textbf{generalized Heisenberg\ algebra} $\mathcal{\hat{A}}$ connected
with a compact Riemann surface $\mathcal{M}$ with genus $g$ and a pair of
points $S_{+}$ and $S_{-}$ in a general position is meant an algebra
generated by generators $A_{n}$ and a central element $1$ with relations
\begin{equation}
\left[ A_{n},A_{m}\right] =\gamma _{nm}1,\text{ }\left[ A_{n},1\right] =0,
\label{'211}
\end{equation}%
where the numbers $\gamma _{nm}$ are defined as%
\begin{equation}
\gamma _{nm}=Res\left( A_{m}\left( P\right) dA_{n}\left( P\right) \right) ,
\label{'213}
\end{equation}%
where $A_{n}\left( P\right) ,n\in Z^{\prime }$ are the KN basis connected
with $S_{+}$ and $S_{-}.$
\end{definition}

For $g=0,$ $S_{\pm }$ are the points $0$ and $\infty ,$ the Lie brackets \ref%
{'211} become%
\begin{equation*}
\left[ A_{n},A_{m}\right] =-n\delta _{n+m,0},\text{ }\left[ A_{n},1\right]
=0.
\end{equation*}%
Hence on a Riemann sphere, the generalized Heisenberg algebra becomes the
ordinary Heisenberg algebra.

\bigskip Now, we construct a representation of the generalized Heisenberg
algebra. Note that $\gamma _{nm}=0$ when $n,m\geq \frac{g}{2},$ then we get
a commutative subalgera $\mathcal{\hat{A}}_{+}$ generated by $A_{n}$ and $1$
with $n\geq \frac{g}{2}$ which has a trivial 1-dimensional representation $%
\mathbb{C}
.$ Hence the algebra has an induced representation of the generalized
Heisenberg algebra%
\begin{equation*}
V=Ind_{\mathcal{\hat{A}}_{+}}^{\mathcal{\hat{A}}}%
\mathbb{C}
=U\left( \mathcal{\hat{A}}\right) \otimes _{U\left( \mathcal{\hat{A}}%
_{+}\right) }%
\mathbb{C}
.
\end{equation*}
Clearly, any element in $\mathcal{\hat{A}}$ can be seen as an endomorphism
of $V.$ Let $|0\rangle =1\otimes 1.$ By the Poincar$\acute{e}$-Birkhoff-Witt
theorem, $V$ has a PBW basis%
\begin{equation}
W=\left\{ \left. A_{-j_{1}+\frac{g}{2}}\cdots A_{-j_{n}+\frac{g}{2}%
}|0\rangle \right\vert j_{1}\geq j_{2}\geq \ldots \geq j_{n}\geq 1,j_{i}\in
\mathbb{N}
\right\} ,  \label{basis of V}
\end{equation}%
and V is a $%
\mathbb{Z}
-$graded vector space by defining the weight of the element $A_{-j_{1}+\frac{%
g}{2}}\cdots A_{-j_{n}+\frac{g}{2}}|0\rangle $ in the basis as $j_{1}+\cdots
+j_{n}.$ Denote the weight of a homogeneous element $v$ by $wt\left(
v\right) .$

\begin{remark}
$V$ is graded only as a vector space, not a representation. In fact, as a
representasion, $V$ is quasi-graded.
\end{remark}

Define the space of states of KN\ vertex algebra by $V,$ the translation
operator by the actions
\begin{equation*}
T|0\rangle =0\text{ \ and }\left[ T,A_{n}\right] =\sum_{u\in Z^{\prime }}\xi
_{n}^{u}A_{u},
\end{equation*}%
where $\xi _{n}^{u}=-Res_{P=S_{+}}\left( \left\langle \omega ^{u}\left(
P\right) ,e_{\frac{3g}{2}-1}\left( P\right) \right\rangle dA_{n}\left(
P\right) \right) ,$ and KN vertex operators by:%
\begin{eqnarray*}
Y\left( |0\rangle ,P\right) &=&Id_{V} \\
Y\left( A_{\frac{g}{2}-1}|0\rangle ,P\right) &=&\sum_{n\in Z^{\prime
}}A_{n}\langle \omega ^{n}\left( P\right) ,e\left( P\right) \rangle \\
Y\left( A_{\frac{g}{2}-m-1}|0\rangle ,P\right) &=&\frac{1}{m!}\nabla
^{m}Y\left( A_{\frac{g}{2}-1}|0\rangle ,P\right) .
\end{eqnarray*}%
For simplicity, denote $Y\left( A_{\frac{g}{2}-1}|0\rangle ,P\right) $ by $%
A\left( P\right) $ sometimes. For elements as $A_{-j_{1}+\frac{g}{2}%
-1}\cdots A_{-j_{n}+\frac{g}{2}-1}|0\rangle $ with $n>1,$ before defining
their KN vertex operators, we need to give the notation of normal ordered
product.

Now we recall the case in a vertex algebras, the normal ordered product of
two fields is defined as%
\begin{equation}
:a\left( w\right) b\left( w\right) :=Res_{z=0}\left[ a\left( z\right)
b\left( w\right) i_{z,w}\frac{1}{z-w}-b\left( w\right) a\left( z\right)
i_{w,z}\frac{1}{z-w}\right] ,  \label{NOP1}
\end{equation}%
that is,%
\begin{equation}
:a\left( w\right) b\left( w\right) :=\sum {}_{n<0}a_{n}w^{-n-1}b\left(
z\right) +b\left( w\right) \sum {}_{n\geq 0}a_{n}w^{-n-1}.  \label{NOP2}
\end{equation}%
Since $i_{P,Q}$ and $i_{Q,P}$ are the extensions of $i_{z,w}$ and $i_{w,z}$
respectively, the simple and natural way is to use some analogous form of Eq.%
\ref{NOP1} to define the normal ordered product. Now, we give the notation
on a pair of one form and $\mu -$form series with the coefficient in $%
End\left( V\right) .$ Let $a\left( Q\right) =\sum a_{n}\omega ^{n}\left(
Q\right) ,$ $b_{\mu }\left( Q\right) =\sum b_{n}f_{\mu }^{n}\left( Q\right)
, $ where $f_{\mu }^{n}\left( Q\right) $ is a differential $\mu -$form.
Define the KN normal ordered product of the formal series $a\left( Q\right) $
and $b_{\mu }\left( Q\right) $ as%
\begin{equation}
:a\left( Q\right) b_{\mu }\left( Q\right) :=Res{}_{P=S_{+}}\left(
i_{P,Q}S_{1}\left( P,Q\right) a\left( P\right) b_{\mu }\left( Q\right)
-i_{Q,P}S_{1}\left( P,Q\right) b_{\mu }\left( Q\right) a\left( P\right)
\right) .  \label{NOPres}
\end{equation}%
By direct calcus, we have
\begin{equation*}
:a\left( Q\right) b_{\mu }\left( Q\right) :=\sum_{n<\frac{g}{2}}a_{n}\omega
^{n}\left( Q\right) b_{\mu }\left( Q\right) +b_{\mu }\left( Q\right)
\sum_{n\geq \frac{g}{2}}a_{n}\omega ^{n}\left( Q\right)
\end{equation*}%
which is similar to the equation \ref{NOP2}. Through the normal ordered
product is defined only on the pair of 1-form and $\lambda -$form series, it
is enough for us.

For the fields $a^{\prime }\left( P\right) =a\left( P\right) \left( e\left(
P\right) \right) $ and $b_{\mu }^{\prime }\left( Q\right) =b_{\mu }\left(
Q\right) \left( e^{\mu }\left( Q\right) \right) ,$ we define their normal
ordered product to be
\begin{equation}
:a^{\prime }\left( Q\right) b_{\mu }^{\prime }\left( Q\right) :=:a\left(
Q\right) b_{\mu }\left( Q\right) :\left( e^{\mu +1}\left( Q\right) \right)
\label{KN-nop}
\end{equation}%
where $e^{\mu }\left( Q\right) =e\left( Q\right) ^{\otimes \mu },$ $a\left(
P\right) \left( e\left( P\right) \right) $ means the coupling of 1-form
series $a\left( P\right) $ and the tangent field $e\left( P\right) ,$ that
is, $\sum_{n\in Z^{\prime }}a_{n}\left\langle \omega ^{n}\left( P\right)
,e\left( P\right) \right\rangle ,$ $b_{\mu }\left( Q\right) \left( e^{\mu
}\left( Q\right) \right) $ and $:a\left( Q\right) b_{\mu }\left( Q\right)
:\left( e^{\mu +1}\left( Q\right) \right) $ have the same meaning. Now we
can define the KN vertex operators as follows\bigskip
\begin{eqnarray*}
Y\left( A_{-j_{1}+\frac{g}{2}-1}\cdots A_{-j_{n}+\frac{g}{2}-1}|0\rangle
,P\right) &:&=Y\left( A_{-j_{1}+\frac{g}{2}-1}|0\rangle ,P\right) \cdots
Y\left( A_{-j_{n}+\frac{g}{2}-1}|0\rangle ,P\right) : \\
&:&=Y\left( A_{-j_{1}+\frac{g}{2}-1}|0\rangle ,P\right) \cdots :Y\left(
A_{-j_{n-1}+\frac{g}{2}-1}|0\rangle ,P\right) Y\left( A_{-j_{n}+\frac{g}{2}%
-1}|0\rangle ,P\right) ::.
\end{eqnarray*}%
It's not difficulty to prove the KN vertex operators are KN\ fields. We make
our focus on proving the axioms of KN vertex algebra.

\subsection{Vacuum axiom, transation axiom}

The statement $Y\left( \left\vert 0\right\rangle ,P\right) =ID$ follows from
our definition. The remainder of the vacuum axiom follows by induction. We
start with the case $A_{\frac{g}{2}-m-1}\left\vert 0\right\rangle $ with $%
m\geq 0.$ Note
\begin{eqnarray*}
Y\left( A_{\frac{g}{2}-m-1}|0\rangle ,P\right) |0\rangle &=&\frac{1}{m!}%
\nabla ^{m}Y\left( A_{\frac{g}{2}-1}|0\rangle ,P\right) \left\vert
0\right\rangle \\
&=&\frac{1}{m!}\sum_{n\leq \frac{g}{2}-1}A_{n}\left\langle \nabla ^{m}\omega
^{n}\left( P\right) ,e\left( P\right) \right\rangle \left\vert 0\right\rangle
\\
&=&\frac{1}{m!}\sum_{n\leq \frac{g}{2}-1}A_{n}\xi _{u_{1}}^{n}\cdots \xi
_{u_{m}}^{u_{m-1}}\left\langle \omega ^{u_{m}}\left( P\right) ,e\left(
P\right) \right\rangle \left\vert 0\right\rangle .
\end{eqnarray*}%
If $A_{n}\xi _{u_{1}}^{n}\cdots \xi _{u_{m-1}}^{u_{m-1}}\not=0,$ then $n\leq
\frac{g}{2}-1-m$ by Lemma \ref{prop183}(3), correspondingly, $u_{m}\leq
\frac{g}{2}-1$ by Lemma\ref{prop183} (2). Hence the KN\ vertex operator
action on $\left\vert 0\right\rangle $ is well defined at $P=S_{+}.$ More
precisely,%
\begin{eqnarray*}
\left. Y\left( A_{\frac{g}{2}-m-1}|0\rangle ,P\right) |0\rangle \right\vert
_{P=S_{+}} &=&\frac{1}{m!}\sum_{n\leq \frac{g}{2}-1-m}\sum_{u_{m}\leq \frac{g%
}{2}-1}\left. A_{n}\xi _{u_{1}}^{n}\cdots \xi _{u_{m}}^{u_{m-1}}z^{-u_{m}+%
\frac{g}{2}-1}\left( 1+o\left( z\right) \right) \left\vert 0\right\rangle
\right\vert _{z=0} \\
&=&\frac{1}{m!}\sum_{n\leq \frac{g}{2}-1-m}A_{n}\xi _{u_{1}}^{n}\cdots \xi _{%
\frac{g}{2}-1}^{u_{m-1}}\left\vert 0\right\rangle .
\end{eqnarray*}%
By Lemma\ref{prop183} (2), the sum above is not zero only when $n=\frac{g}{2}%
-1-m.$ Using Lemma\ref{prop183} (1), we got $\left. Y\left( A_{\frac{g}{2}%
-m-1}|0\rangle ,P\right) |0\rangle \right\vert _{P=S_{+}}=A_{\frac{g}{2}%
-m-1}|0\rangle .$

Define the level of the element $A_{\frac{g}{2}-n_{m}-1}\cdots A_{\frac{g}{2}%
-n_{1}-1}\left\vert 0\right\rangle $\ by $m.$ Then for any element of level
one, we have proved its KN vertex operator satisfies the vacuum axiom. For
the higher levels, we will prove it by induction as follows.

\begin{theorem}
\label{162}For any element $a\in V,$ $Y\left( a,P\right) |0\rangle $ is well
defined at $P=S_{+}.$ Moreover, if \ a is a homogeneous element in $V$ with
weight $m,$ then $Y\left( a,P\right) |0\rangle |_{P=S_{+}}\equiv a$ $mod $ $%
\oplus _{k<m}V_{k}|0\rangle .$
\end{theorem}

\begin{proof}
Assume that it is true for anyelement of level $m$ in the basis $W.$ That
is, for any element $A_{\frac{g}{2}-n_{m}-1}\cdots A_{\frac{g}{2}%
-n_{1}-1}\left\vert 0\right\rangle $ in $V$ with $n_{m}\geq n_{m-1}\geq
\cdots \geq n_{1}\geq 0,$ we have
\begin{equation*}
Y\left( A_{\frac{g}{2}-n_{m}-1}\cdots A_{\frac{g}{2}-n_{1}-1}\left\vert
0\right\rangle ,P\right) |0\rangle |_{P=S_{+}}=A_{\frac{g}{2}-n_{m}-1}\cdots
A_{\frac{g}{2}-n_{1}-1}\left\vert 0\right\rangle +v|0\rangle ,\text{ }
\end{equation*}%
where $v\in \oplus _{k<n_{1}+\cdots +n_{m}+m}V_{k}.$

For the element $A_{\frac{g}{2}-n-1}A_{\frac{g}{2}-n_{m}-1}\cdots A_{\frac{g%
}{2}-n_{1}-1}\left\vert 0\right\rangle $ with $n\geq n_{m}\geq \cdots \geq
n_{1}\geq 0$, by the definition of KN\quad normal ordered product and
Proposition \ref{lie derivative on omiga}, we have
\begin{eqnarray*}
&&Y\left( A_{\frac{g}{2}-n-1}A_{\frac{g}{2}-n_{m}-1}\cdots A_{\frac{g}{2}%
-n_{1}-1}\left\vert 0\right\rangle ,P\right) \\
&=&\frac{1}{n!}:\nabla ^{n}A\left( P\right) Y\left( A_{\frac{g}{2}%
-n_{m}-1}\cdots A_{\frac{g}{2}-n_{1}-1}\left\vert 0\right\rangle ,P\right) :
\\
&=&\frac{1}{n!}:A_{k}\xi _{u_{1}}^{k}\xi _{u_{2}}^{u_{1}}\cdots \xi
_{u_{n}}^{u_{n-1}}\left\langle \omega ^{u_{n}}\left( P\right) ,e\left(
P\right) \right\rangle Y\left( A_{\frac{g}{2}-n_{m}-1}\cdots A_{\frac{g}{2}%
-n_{1}-1}\left\vert 0\right\rangle ,P\right) : \\
&=&Y\left( A_{\frac{g}{2}-n_{m}-1}\cdots A_{\frac{g}{2}-n_{1}-1}\left\vert
0\right\rangle ,P\right) \sum_{u_{n}\geq \frac{g}{2}}A_{k}\frac{\xi
_{u_{1}}^{k}\xi _{u_{2}}^{u_{1}}\cdots \xi _{u_{n}}^{u_{n-1}}}{n!}%
\left\langle \omega ^{u_{n}}\left( P\right) ,e\left( P\right) \right\rangle
\\
&&+\sum_{u_{n}<\frac{g}{2}}A_{k}\frac{\xi _{u_{1}}^{k}\xi
_{u_{2}}^{u_{1}}\cdots \xi _{u_{n}}^{u_{n-1}}}{n!}\left\langle \omega
^{u_{n}}\left( P\right) ,e\left( P\right) \right\rangle Y\left( A_{\frac{g}{2%
}-n_{m}-1}\cdots A_{\frac{g}{2}-n_{1}-1}\left\vert 0\right\rangle ,P\right) .
\end{eqnarray*}%
For $u_{n}\geq \frac{g}{2},$ if $\xi _{u_{1}}^{k}\xi _{u_{2}}^{u_{1}}\cdots
\xi _{u_{n}}^{u_{n-1}}\not=0,$ then by Lemma \ref{prop183}$(2),$ $k\geq
u_{n}-n\geq \frac{g}{2}-n.$ Using the lemma(3), we know $k$ should be bigger
than $\frac{g}{2}.$ Then the first sum kills $\left\vert 0\right\rangle ,$
and by our inductive assumption, the second sum gives a series with positive
powers of $z$ with the constant term%
\begin{equation*}
A_{k}\frac{\xi _{u_{1}}^{k}\xi _{u_{2}}^{u_{1}}\cdots \xi _{\frac{g}{2}%
-1}^{u_{n-1}}}{n!}\left( A_{\frac{g}{2}-n_{m}-1}\cdots A_{\frac{g}{2}%
-n_{1}-1}\left\vert 0\right\rangle +v\left\vert 0\right\rangle \right) .
\end{equation*}%
If $\xi _{u_{1}}^{k}\xi _{u_{2}}^{u_{1}}\cdots \xi _{\frac{g}{2}%
-1}^{u_{n-1}}\not=0,$ then $k\geq \frac{g}{2}-n-1$ by Lemma\ref{prop183}(2),
correspondingly, $k\geq \frac{g}{2}$ or $k=\frac{g}{2}-n-1$ by the third
part of the lemma. Hence%
\begin{eqnarray*}
&&Y\left( A_{\frac{g}{2}-n-1}A_{\frac{g}{2}-n_{m}-1}\cdots A_{\frac{g}{2}%
-n_{1}-1}\left\vert 0\right\rangle ,P\right) |0\rangle |_{P=S_{+}} \\
&=&\left( \frac{\xi _{u_{1}}^{\frac{g}{2}-n-1}\xi _{u_{2}}^{u_{1}}\cdots \xi
_{\frac{g}{2}-1}^{u_{n-1}}}{n!}A_{\frac{g}{2}-n-1}+\sum_{k\geq \frac{g}{2}%
}a^{k}A_{k}\right) \left( A_{\frac{g}{2}-n_{m}-1}\cdots A_{\frac{g}{2}%
-n_{1}-1}\left\vert 0\right\rangle +v|0\rangle \right) \\
&=&A_{\frac{g}{2}-n-1}\cdots A_{\frac{g}{2}-n_{1}-1}\left\vert
0\right\rangle +A_{\frac{g}{2}-n-1}v|0\rangle +\sum_{k\geq \frac{g}{2}%
-n}a^{k}A_{k}\left( A_{\frac{g}{2}-n_{m}-1}\cdots A_{\frac{g}{2}%
-n_{1}-1}\left\vert 0\right\rangle +v|0\rangle \right) ,
\end{eqnarray*}%
where $a^{k}=\frac{\xi _{u_{1}}^{k}\xi _{u_{2}}^{u_{1}}\cdots \xi _{\frac{g}{%
2}-1}^{u_{_{n-1}}}}{n!}.$ Here, the last equation is got by Lemma\ref%
{prop183}(3). Because weights of homogenous elements in $A_{\frac{g}{2}%
-n}v\left\vert 0\right\rangle $ and $A_{k}\left( A_{\frac{g}{2}%
-n_{m}-1}\cdots A_{\frac{g}{2}-n_{1}-1}\left\vert 0\right\rangle +v|0\rangle
\right) $ are smaller than $n_{1}+\cdots +n_{m}+n+m+1,$ then
\begin{eqnarray*}
&&Y\left( A_{\frac{g}{2}-n-1}A_{\frac{g}{2}-n_{m}-1}\cdots A_{\frac{g}{2}%
-n_{1}-1}\left\vert 0\right\rangle ,P\right) |0\rangle |_{P=S_{+}} \\
&=&A_{\frac{g}{2}-n-1}A_{\frac{g}{2}-n_{m}-1}\cdots A_{\frac{g}{2}%
-n_{1}-1}\left\vert 0\right\rangle \text{ }mod\text{ }\left( \oplus
_{k<n_{1}+\cdots +n_{m}+n+m+1}V_{k}\right) |0\rangle .
\end{eqnarray*}%
Hence for any homogenegous element $a$ with level $m+1,$ the vacuum axiom is
satisfied.
\end{proof}

\bigskip Now, we prove the translation axiom.

For $\left\vert 0\right\rangle \in V_{0},$ from $T|0\rangle =0$ and $Y\left(
|0\rangle ,P\right) =Id_{V},$ we have $\left[ T,Y\left( |0\rangle ,P\right) %
\right] =\nabla Y\left( |0\rangle ,P\right) .$

For $A_{\frac{g}{2}-1}\left\vert 0\right\rangle \in V_{1},$ we have
\begin{eqnarray}
\left[ T,A_{\frac{g}{2}-1}\left\vert 0\right\rangle \right] &=&\sum_{n\in
Z^{\prime }}\left[ T,A_{n}\right] \left\langle \omega ^{n}\left( P\right)
,e\left( P\right) \right\rangle  \notag \\
&=&\sum_{n,u\in Z^{\prime }}\left( \xi _{n}^{u}A_{u}\right) \left\langle
\omega ^{n}\left( P\right) ,e\left( P\right) \right\rangle  \notag \\
&=&\sum_{u\in Z^{\prime }}A_{u}\left\langle \nabla \omega ^{u}\left(
P\right) ,e\left( P\right) \right\rangle  \notag \\
&=&\nabla A\left( P\right) ,  \label{251}
\end{eqnarray}%
where the first and second equation are got by the fact $\left[ T,A_{n}%
\right] =\sum_{u\in Z^{\prime }}\xi _{n}^{u}A_{u}$ and the equation $\nabla
\omega ^{u}\left( P\right) =\xi _{n}^{u}\omega ^{n}\left( P\right) $
respectively. For $A_{-n+\frac{g}{2}-1}|0\rangle \in V_{n+1},$ from
\begin{equation*}
\left[ T,Y\left( A_{-n+\frac{g}{2}-1}|0\rangle ,P\right) \right] =\frac{1}{n!%
}\left[ T,\nabla ^{n}A\left( P\right) \right] =\frac{1}{n!}\nabla ^{n}\left[
T,A\left( P\right) \right]
\end{equation*}%
and the equation $\left( \ref{251}\right) ,$ we get%
\begin{equation*}
\left[ T,Y\left( A_{-n+\frac{g}{2}-1}|0\rangle ,P\right) \right] =\frac{1}{n!%
}\nabla ^{n+1}A\left( P\right) =\nabla Y\left( A_{-n+\frac{g}{2}-1}|0\rangle
,P\right) .
\end{equation*}

Hence, for any $V_{n}$ with $n\in
\mathbb{Z}
_{+},$ there exists an element $a$ in $V_{n}$ such that
\begin{equation*}
\left[ T,Y\left( a,P\right) \right] =\nabla Y\left( a,P\right) .
\end{equation*}

\subsection{Locality axiom}

Recall the state space has a basis%
\begin{equation*}
W=\left\{ A_{\frac{g}{2}-n_{1}-1}\cdots A_{\frac{g}{2}-n_{k}-1}|0\rangle |%
\text{ }k\geq 0,\text{ }n_{1}\geq \cdots \geq n_{k}\geq 0\right\} ,
\end{equation*}%
and the element $A_{\frac{g}{2}-n_{1}-1}\cdots A_{\frac{g}{2}%
-n_{k}-1}|0\rangle $ in $W$ has level $k.$

Note that for any element $a$ in $W$ with level $k,$ $Y\left( a,P\right) $
has the form $\sum a_{n_{1}\cdots n_{k}}\omega ^{n_{1}}\left( P\right)
\cdots \omega ^{n_{k}}\left( P\right) \left( e^{k}\left( P\right) \right) .$
If we denote $\tilde{Y}\left( a,P\right) =\sum a_{n_{1}\cdots n_{k}}\omega
^{n_{1}}\left( P\right) \cdots \omega ^{n_{k}}\left( P\right) ,$ then $%
Y\left( a,P\right) =\tilde{Y}\left( a,P\right) \left( e^{k}\left( P\right)
\right) .$ Since for arbitrary two elements $a,b$ in $W$ with $level\left(
a\right) =k,level\left( b\right) =k^{\prime },$%
\begin{equation*}
\left[ Y\left( a,P\right) ,Y\left( b,P\right) \right] =\left[ \tilde{Y}%
\left( a,P\right) ,\tilde{Y}\left( b,Q\right) \right] \left( e^{k}\left(
P\right) \otimes e^{k^{\prime }}\left( Q\right) \right) ,
\end{equation*}%
then the\ locality axiom holds on the generalized Heisenberg KN vertex
algebra if the following equalities are true
\begin{equation}
F_{u}^{N}\left( P,Q\right) \left[ \tilde{Y}\left( a,P\right) ,\tilde{Y}%
\left( b,Q\right) \right] =0,\text{ }\forall u\in Z^{\prime },N\gg 0.
\label{star}
\end{equation}

In the following, we will prove that for any $a,b\in W,$ $\tilde{Y}\left(
a,P\right) ,\tilde{Y}\left( b,Q\right) $ satisy the condition $\left( \ref%
{star}\right) $ which is also called locality.

Firstly, we check that $\left( \ref{star}\right) $ holds for $a=b=A_{\frac{g%
}{2}-1}.$ Since

\begin{equation*}
\left[ \tilde{Y}\left( A_{\frac{g}{2}-1}|0\rangle ,P\right) ,\tilde{Y}\left(
A_{\frac{g}{2}-1}|0\rangle ,Q\right) \right] =\left[ A_{n},A_{m}\right]
\omega ^{n}\left( P\right) \omega ^{m}\left( Q\right) =\gamma _{nm}\omega
^{n}\left( P\right) \omega ^{m}\left( Q\right)
\end{equation*}%
and
\begin{equation*}
d_{P}\triangle \left( P,Q\right) =\left[ dA_{m}\left( P\right) \right]
\omega ^{m}\left( Q\right) =\gamma _{mn}\omega ^{n}\left( P\right) \omega
^{m}\left( Q\right) =-\gamma _{nm}\omega ^{n}\left( P\right) \omega
^{m}\left( Q\right) ,
\end{equation*}%
where $\gamma _{nm}=Res\left( A_{m}\left( P\right) dA_{n}\left( P\right)
\right) ,$ then $\left[ \tilde{Y}\left( A_{\frac{g}{2}-1}|0\rangle ,P\right)
,\tilde{Y}\left( A_{\frac{g}{2}-1}|0\rangle ,Q\right) \right]
=-d_{P}\triangle \left( P,Q\right) .$ By Formula $\left( \ref{1122}\right) ,$
we have%
\begin{equation*}
F_{u}^{2}\left( P,Q\right) \left[ \tilde{Y}\left( A_{\frac{g}{2}-1}|0\rangle
,P\right) ,\tilde{Y}\left( A_{\frac{g}{2}-1}|0\rangle ,Q\right) \right] =0,%
\text{ }\forall u\in Z^{\prime }.
\end{equation*}%
According to Theorem $\left( \ref{214}\right) $ and
\begin{equation*}
\left[ \tilde{Y}\left( A_{\frac{g}{2}-n-1}|0\rangle ,P\right) ,\tilde{Y}%
\left( A_{\frac{g}{2}-m-1}|0\rangle ,Q\right) \right] =\frac{1}{n!m!}\left[
\nabla ^{n}\tilde{Y}\left( A_{\frac{g}{2}-1}|0\rangle ,P\right) ,\nabla ^{m}%
\tilde{Y}\left( A_{\frac{g}{2}-1}|0\rangle ,Q\right) \right] =\frac{1}{n!m!}%
\nabla _{P}^{n}\nabla _{Q}^{m}d_{P}\triangle \left( P,Q\right) ,
\end{equation*}%
we get $\tilde{Y}\left( A_{\frac{g}{2}-n-1}|0\rangle ,P\right) $ and $\tilde{%
Y}\left( A_{\frac{g}{2}-m-1}|0\rangle ,Q\right) $ are local.

Finally, locality of any two formal series \bigskip $\tilde{Y}\left(
a,P\right) $ and $\tilde{Y}\left( b,Q\right) $ with $a,b\in W$ follows by an
induction from locality of $\tilde{Y}\left( A_{\frac{g}{2}-n-1}|0\rangle
,P\right) $ and $\tilde{Y}\left( A_{\frac{g}{2}-m-1}|0\rangle ,Q\right) $
using the following theorem. Hence the locality of any two fields can be got
because $W$ is a basis of the state space.

\begin{theorem}[Generalized Dong's Lemma]
\label{Modified Dong's Lemma}Let $g_{\mu }^{m}\left( P\right) =\omega
^{m_{1}}\left( P\right) \cdots \omega ^{m_{\mu }}\left( P\right) $ with $%
m_{i}\in Z^{\prime }.$ If $a\left( P\right) =\sum_{n\in Z^{\prime
}}a_{n}\omega ^{n}\left( P\right) ,$ $b_{\mu }\left( P\right) =\sum_{m\in
Z^{\prime }}b_{m}g_{\mu }^{m}\left( P\right) ,$ $c_{\gamma }\left( P\right)
=\sum_{n\in Z^{\prime }}c_{n}g_{\gamma }^{n}\left( P\right) $ are mutual
local, then $:a\left( P\right) b_{\mu }\left( P\right) :$ and $c_{\gamma
}\left( P\right) $ are mutual local in $\mathcal{M\times M}.$
\end{theorem}

\begin{proof}
By assumption, we may find $N$ so that for all $m\geq N,$
\begin{eqnarray}
F_{u}\left( P,Q\right) ^{m}a\left( P\right) b_{\mu }\left( Q\right)
&=&F_{u}\left( P,Q\right) ^{m}b_{\mu }\left( Q\right) a\left( P\right) ,
\label{formula'221} \\
F_{u}\left( P,Q\right) ^{m}b_{\mu }\left( P\right) c_{\gamma }\left(
Q\right) &=&F_{u}\left( P,Q\right) ^{m}c_{\gamma }\left( Q\right) b_{\mu
}\left( P\right) ,  \label{formula'222} \\
F_{u}\left( P,Q\right) ^{m}a\left( P\right) c_{\gamma }\left( Q\right)
&=&F_{u}\left( P,Q\right) ^{m}c_{\gamma }\left( Q\right) a\left( P\right) .
\label{formula'223}
\end{eqnarray}%
We wish to find an integer $M$ such that
\begin{equation*}
F_{u}\left( Q,R\right) ^{M}:a\left( Q\right) b_{\mu }\left( Q\right)
:c_{\gamma }\left( R\right) =F_{u}\left( Q,R\right) ^{M}:a\left( Q\right)
b_{\mu }\left( Q\right) :c_{\gamma }\left( R\right) .
\end{equation*}%
By Formula$\left( \ref{NOPres}\right) ,$ this will follow from the statement%
\begin{eqnarray}
&&F_{u}\left( Q,R\right) ^{M}\left[ -i_{P,Q}S_{1}\left( P,Q\right) a\left(
P\right) b_{\mu }\left( Q\right) +i_{Q,P}S_{1}\left( P,Q\right) b_{\mu
}\left( Q\right) a\left( P\right) \right] c_{\gamma }\left( R\right)
\label{2.3.6} \\
&=&F_{u}\left( Q,R\right) ^{M}c_{\gamma }\left( R\right) \left(
-i_{P,Q}S_{1}\left( P,Q\right) a\left( P\right) b_{\mu }\left( Q\right)
+i_{Q,P}S_{1}\left( P,Q\right) b_{\mu }\left( Q\right) a\left( P\right)
\right) .  \notag
\end{eqnarray}%
Let us take $M=3N.$ Writing $F_{u}\left( Q,R\right) ^{3N}=F_{u}\left(
Q,R\right) ^{N}\sum_{k=0}^{2N}C_{2N}^{k}F_{u}\left( Q,P\right)
^{k}F_{u}\left( P,R\right) ^{2N-k},$ we see that in the terms in the left
hand side of $\left( \ref{2.3.6}\right) $ with $N<k\leq 2N$ vanish because
\begin{eqnarray}
&&F_{u}\left( Q,P\right) ^{k}\left[ -i_{P,Q}S_{1}\left( P,Q\right) a\left(
P\right) b_{\mu }\left( Q\right) +i_{Q,P}S_{1}\left( P,Q\right) b_{\mu
}\left( Q\right) a\left( P\right) \right]  \notag \\
&=&F_{u}\left( Q,P\right) ^{k-\left( N+1\right) }F_{u}\left( Q,P\right)
^{\left( N+1\right) }\left[ -i_{P,Q}S_{1}\left( P,Q\right) a\left( P\right)
b_{\mu }\left( Q\right) +i_{Q,P}S_{1}\left( P,Q\right) b_{\mu }\left(
Q\right) a\left( P\right) \right]  \notag \\
&=&F_{u}\left( Q,P\right) ^{k-\left( N+1\right) }F_{u}\left( Q,P\right)
\left[ -i_{P,Q}S_{1}\left( P,Q\right) +i_{Q,P}S_{1}\left( P,Q\right) \right]
F_{u}^{N}\left( Q,P\right) a\left( P\right) b_{\mu }\left( Q\right)  \notag
\\
&=&F_{u}\left( Q,P\right) ^{k-\left( N+1\right) }\left[ F_{u}\left(
Q,P\right) \triangle \left( P,Q\right) \right] \left[ a\left( P\right)
b_{\mu }\left( Q\right) F_{u}\left( Q,P\right) ^{N}\right]  \notag \\
&=&0,  \label{'2419}
\end{eqnarray}%
where the last equality is got by the fact $\triangle \left( P,Q\right)
F_{u}\left( P,Q\right) =0.$ For the terms with $0\leq k\leq N,$ by $\left( %
\ref{formula'221}\right) -\left( \ref{formula'223}\right) ,$ we get%
\begin{eqnarray}
&&F_{u}\left( Q,R\right) ^{N}F_{u}\left( P,R\right) ^{k}\left[
-i_{P,Q}S_{1}\left( P,Q\right) a\left( P\right) b_{\mu }\left( Q\right)
+i_{Q,P}S_{1}\left( P,Q\right) b_{\mu }\left( Q\right) a\left( P\right) %
\right] c_{\gamma }\left( R\right)  \notag \\
&=&F_{u}\left( Q,R\right) ^{N}F_{u}\left( P,R\right) ^{k}c_{\gamma }\left(
R\right) \left[ -i_{P,Q}S_{1}\left( P,Q\right) a\left( P\right) b_{\mu
}\left( Q\right) +i_{Q,P}S_{1}\left( P,Q\right) b_{\mu }\left( Q\right)
a\left( P\right) \right] .  \label{'251}
\end{eqnarray}%
The same phenomena occures on the right side of $\left( \ref{2.3.6}\right) :$
the terms with $N\leq k\leq 2N$ will vanish, and the other terms give us the
same expression as what we now have on the left hand side. Thus we have
establish $\left( \ref{2.3.6}\right) ,$ and hence the theorem.
\end{proof}

\subsection{Generalized Heisenberg KN vertex algebras on Riemann spheres}

Recall that on a Riemann sphere, choose $S_{+}=0$ and $S_{-}=\infty ,$ then $%
A_{n}\left( z\right) =z^{n},\omega ^{n}\left( z\right) =z^{-n-1}dz,e\left(
z\right) =\frac{\partial }{\partial z}$ with $z\in
\mathbb{C}
\cup \{\infty \}.$ Since
\begin{eqnarray*}
\gamma _{nm} &=&-Res_{P=S_{+}}\left( A_{n}\left( z\right) dA_{m}\left(
z\right) \right) \\
&=&-Res_{z=0}\left( mz^{n+m-1}dz\right) \\
&=&-m\delta _{n}^{-m}=n\delta _{n}^{-m},
\end{eqnarray*}%
then the Generalized Heisenberg algebra on $%
\mathbb{C}
P^{1}$ is generated by $A_{n}$ and a central element $1$ with relations%
\begin{equation*}
\left[ A_{n},A_{m}\right] =n\delta _{n}^{-m}1,\text{ }\left[ A_{n},1\right]
=0,
\end{equation*}%
which is indeed the structure of the Heisenberg algebra.

Since $e\left( z\right) =\partial _{z},$ then $\nabla _{e\left( z\right) }$
is $\partial _{z}$ when acts on meromorphic functions. Corresponding, the
Lie derivative on a field $a\left( z\right) =\sum_{n\in
\mathbb{Z}
}a_{n}A_{n}\left( z\right) $ is $\nabla _{e\left( z\right) }a\left( z\right)
=\sum_{n\in
\mathbb{Z}
}a_{n}\partial _{z}\left( A_{n}\left( z\right) \right) .$

Because $Z^{\prime }=%
\mathbb{Z}
+\frac{g}{2}=%
\mathbb{Z}
$ and $\left\langle \omega ^{n}\left( z\right) ,e\left( z\right)
\right\rangle =z^{-n-1},$ then the vertex operators are defined as
\begin{eqnarray*}
Y\left( A_{-1}\left\vert 0\right\rangle ,z\right) &=&\sum_{n\in
\mathbb{Z}
}A_{n}\left\langle \omega ^{n}\left( z\right) ,e\left( z\right)
\right\rangle =\sum_{n\in
\mathbb{Z}
}A_{n}z^{-n-1} \\
Y\left( A_{-m-1}\left\vert 0\right\rangle ,z\right) &=&\frac{1}{m!}%
\sum_{n\in
\mathbb{Z}
}A_{n}\nabla ^{m}\left\langle \omega ^{n}\left( z\right) ,e\left( z\right)
\right\rangle \\
&=&\frac{1}{m!}\sum_{n\in
\mathbb{Z}
}A_{n}\partial _{z}^{m}z^{-n-1}=\frac{1}{m!}\partial _{z}^{m}Y\left(
A_{-1}\left\vert 0\right\rangle ,z\right) ,
\end{eqnarray*}%
and
\begin{equation*}
Y\left( A_{-n_{1}-1}\cdots A_{-n_{m}-1}\left\vert 0\right\rangle ,z\right)
=:Y\left( A_{-n_{1}-1}\left\vert 0\right\rangle ,z\right) \cdots :Y\left(
A_{-n_{m-1}-1}\left\vert 0\right\rangle ,z\right) Y\left(
A_{-n_{m}-1}\left\vert 0\right\rangle ,z\right) ::.
\end{equation*}

Next, we proceed to consider the form of KN normal ordered product. Since%
\begin{equation*}
Y\left( A_{-m-1}\left\vert 0\right\rangle ,z\right) =\sum_{n\in
\mathbb{Z}
}\frac{1}{m!}A_{n}\nabla ^{m}\omega ^{n}\left( z\right) ,e\left( z\right)
b\left( z\right) =\sum_{n\in
\mathbb{Z}
}\tbinom{-n}{m}A_{n}\left\langle \omega ^{n+m}\left( z\right) ,e\left(
z\right) \right\rangle b\left( z\right) :,
\end{equation*}%
then by the definition of KN normal ordered product, for any KN\ field $%
b\left( z\right) ,$%
\begin{eqnarray*}
&:&Y\left( A_{-m-1}\left\vert 0\right\rangle ,z\right) b\left( z\right) : \\
&=&\sum_{n+m\leq -1}\tbinom{-n}{m}A_{n}\left\langle \omega ^{n+m}\left(
z\right) ,e\left( z\right) \right\rangle b\left( z\right) +b\left( z\right)
\sum_{n+m\geq 0}\tbinom{-n}{m}\left\langle \omega ^{n+m}\left( z\right)
,e\left( z\right) \right\rangle A_{n} \\
&=&\sum_{n+m\leq -1}\tbinom{-n}{m}A_{n}z^{-n-m-1}b\left( z\right) +b\left(
z\right) \sum_{n+m\geq 0}\tbinom{-n}{m}z^{-n-m-1}A_{n} \\
&=&Y\left( A_{-m-1}\left\vert 0\right\rangle ,z\right) _{+}b\left( z\right)
+b\left( z\right) Y\left( A_{-m-1}\left\vert 0\right\rangle ,z\right) _{-},
\end{eqnarray*}%
For any KN field $b\left( z\right) ,$ where $Y\left( a,z\right)
_{+}:=\sum_{n<0}a_{n}z^{-n-1},$ $Y\left( a,z\right) _{-}:=\sum_{n\geq
0}a_{n}z^{-n-1}.$ Clearly the KN normal ordered product is same as normal
ordered product in the vertex algebra. Hence KN vertex operators of the KN
vertex algebra on a Riemann sphere are same as in the Heisenberg vertex
algebra.

Recall the transition operator $T$ is defined by the actions $T\left\vert
0\right\rangle =0$ and $\left[ T,A_{m}\right] =\sum_{n\in
\mathbb{Z}
}\xi _{m}^{n}A_{n}.$ Since, on a Riemann sphere, $\xi _{m}^{n}=-Res\left(
z^{-n-1}dz^{m}\right) =-m\delta _{m-1}^{n},$ then $T$ is same as in the
vertex algebra.

Hence, the generalized Heisenberg KN vertex algebra on a Riemann sphere is
same as the Heisenberg vertex algebra.

\begin{acknowledgement}
The first author would like to thank Professor Chongying Dong for his
various kinds of help when she researched in University of California, Santa
Cruz.
\end{acknowledgement}

\end{document}